\documentclass[dvipsnames,12pt]{amsart}

\usepackage{amssymb,mathtools,bbm}
\usepackage{color}
\usepackage[normalem]{ulem}
\usepackage{tikz}

\newcommand{\beq}{\begin{equation}}
\newcommand{\eeq}{\end{equation}}
\numberwithin{equation}{section}

\usepackage{enumitem,xcolor}

\newtheorem{thm}{Theorem}[section]
\newtheorem{prop}[thm]{Proposition}
\newtheorem{lemma}[thm]{Lemma}

\newtheorem{thmA}{Theorem}

\theoremstyle{remark}

\theoremstyle{definition}

\newtheorem{defn}[thm]{Definition}
\newtheorem{example}[thm]{Example}

\DeclareMathOperator{\ldens}{\underbar{dens}} % lower density
\DeclareMathOperator{\rhodens}{\rho-\underbar{dens}} % lower rho density
\DeclareMathOperator{\bigO}{{\rm O}}

\newcommand{\A}{\mathcal{A}}

\newcommand{\C}{\mathbb{C}}
\newcommand{\R}{\mathbb{R}}
\newcommand{\N}{\mathbb{N}}
\newcommand{\Z}{\mathbb{Z}}
\newcommand{\calS}{\mathcal{S}}
\newcommand{\T}{\mathcal{T}}
\newcommand{\F}{\mathcal{F}}

\newcommand{\w}{{\bf w}}
\newcommand{\Bw}{B_{\bf w}}
\newcommand{\1}{\mathbbm{1}}  	% indicator function

\newcommand{\abs}[1]{\left| #1 \right|}

\newcommand{\norm}[1]{{\left\|#1\right\|}}

\newcommand{\returnSet}[3]{\mathcal{N}_{#1}(#2, \, #3)}  % return set general

\newcommand{\nT}[1]{\returnSet{#1}{x}{U}}  % return set shortcut
\newcommand{\nTU}{\returnSet{T}{x}{U}}  % return set shortcut

\usepackage{pifont}

\title[$\rho$-frequently hypercyclic operators]{{\large{$\rho$}}-frequently hypercyclic operators} 
\author[T. Carroll]{Tom Carroll}                                
\address{School of Mathematical Sciences, University College Cork, Ireland}                                                                    
\email{t.carroll@ucc.ie}                                                                          
\author[C. Gilmore]{Clifford Gilmore}                                
\address{Laboratoire de Mathématiques Blaise Pascal UMR 6620, Université Clermont Auvergne, Campus universitaire des Cézeaux, 3 place Vasarely, 63178 Aubière Cedex, France}                                                              
\email{clifford.gilmore@uca.fr}                                                                                                   
\date{}                                       
                                                   

\thanks{C.~Gilmore was supported by the European Union’s Horizon Europe research and innovation programme under the Marie Skłodowska-Curie grant agreement No.~101066064. This work was carried out in part during a visit by T.~Carroll to UCA. He is grateful to the members of LMBP for  their  hospitality and stimulating mathematical discussions.}

\usepackage{calc}

\usepackage[a4paper]{geometry}

\newgeometry{twoside=false,,textheight=\textheight+8ex,left=10em,right=8.6em}

\begin{document}

\begin{abstract}
The concept of $\rho$-frequent hypercyclicity is introduced in order to provide a refined form of  frequent hypercyclicity. This is achieved by replacing the denominator in the definition of frequent hypercyclicity by an appropriately chosen calibration function $\rho$.
A $\rho$-Frequent Hypercyclicity Criterion is determined and the $\rho$-frequent hypercyclicity of weighted backward shifts is investigated.
\end{abstract}

\maketitle

%
%
% Introduction
%
%
\section{Introduction}\label{sec:intro}

 A foundational concept in linear dynamics is \textsl{hypercyclicity\/}. 
A continuous linear operator $T$ acting on an infinite dimensional  separable Fr\'echet space is 
\textsl{hypercyclic\/}  if the orbit of some element is dense in the space. 
This natural concept, that can be thought of as \lq linear chaos\rq, 
is equivalent to \textsl{topological transitivity\/}, 
meaning that the orbit of any nonempty open set under $T$ intersects any 
other open set in the space - this is Birkhoff's Transitivity Theorem
(cf.\ \cite[Theorem~1.2]{BM09}). 
An element in the space whose orbit under $T$ is dense is 
called a \textsl{hypercyclic vector\/} for $T$. 
The books by Bayart and Matheron~\cite{BM09} and by Grosse-Erdmann 
and Peris~\cite{LinearChaos2011} remain excellent 
references for the subject, while the review article \cite{CliffReview}  
points also to more recent developments. 

It is natural to ask how often the orbit of a hypercyclic vector may visit 
an open set.  
In this regard, the most that one can expect is that the orbit spends 
a positive fraction of time in each open subset of the space.
This concept, known as \textsl{frequent hypercyclicity\/}, was introduced by 
Bayart  and Grivaux~\cite{BG06}. 
To be precise, for an operator $T$ acting on a  Fr\'echet space $X$, for $x \in X$
and for a nonempty open subset $U$ of $X$, the \textsl{return set\/} is defined to be 
\[
\nTU \coloneqq \big\{ n \in \N \,\colon\, T^nx \in U \big\}.
\]
The return set consists of those times $n$ when the orbit of $x$ under $T$ lies in $U$.  
The  \textsl{lower density\/} of the subset $\nTU$  in the set of natural 
numbers $\N$ is 
\beq\label{density}
\ldens(\nTU) \coloneqq \liminf_{N\to\infty} \frac{\#\{n \in \nTU \,\colon\,n \leq N\}}{N}.
\eeq 
When $X$ is separable and infinite dimensional, the operator $T$ is said to be \textsl{frequently hypercyclic\/} if there is some $x$ 
in $X$ for which $\ldens(\nTU)$ is positive for every nonempty open subset $U$ of $X$. 
Many important and classical operators are known to be frequently hypercyclic. 
These include the differentiation and  translation operators acting on the space 
$\mathcal{H}(\C)$ of entire functions in the complex plane.

We introduce here a refined form of frequent hypercyclicity by replacing 
the denominator in \eqref{density} by $\rho(N)$, where $\rho(t)$ is a 
function that grows slower than $t$. 
Though the concept overlaps with ideas already well-developed 
in the literature, as will be explained in detail later, 
the formulation presented here may have some advantages. 

The properties required of the \lq calibration function\rq\ $\rho$ and 
some straightforward preliminary results for sets with positive lower $\rho$-density 
are set out in Section~\ref{sec:generalsets}. 
For example, the lower $\rho$-density of a set $\A \subseteq \N$ is interpreted in 
terms of the growth of the sequence of natural numbers that makes up $\A$. 
The concept of  $\rho$-frequent hypercyclicity is formally defined
in Section~\ref{sec:rhofhc}, and a version of the Frequent Hypercyclicity Criterion 
adapted to  $\rho$-frequent hypercyclicity is derived. 
In Section~\ref{sec:shifts}, the $\rho$-Frequent Hypercyclicity Criterion leads to 
a sufficient condition for a weighted backward shift on the sequence space 
$\ell^p(\N)$ to be  $\rho$-frequently hypercyclic. 
The necessity of this condition is established for a restricted class of weighted 
backward shifts. 
This, in turn, permits the construction in Section~\ref{sec:varying_rho}
of examples to show that different calibration functions give rise to distinct 
notions of $\rho$-frequent hypercyclicity.
In the final Section~\ref{sec:literature}, we survey several definitions and associated 
results from the literature that overlap with the idea of 
$\rho$-frequent hypercyclicity. 
%
%
% End of Introduction
%
%

\section{Densities relative to a calibration function $\rho$.}
\label{sec:generalsets}

\subsection{Calibration Functions and Density}
\label{subsection:2.1}

We first list the properties required of a calibration function $\rho$. 
\begin{defn}\label{defn:densityfunction}
A function $\rho \colon \R^+ \to \R^+$ is called a \textsl{calibration function} 
if it has the following properties:
\begin{enumerate}
\item $\rho$ is continuous, strictly increasing and unbounded;%
\label{rho_prop1}
\item $\rho$ is  ${\rm Lip}(1,1)$ so that 
$\vert \rho(t_1) - \rho(t_2) \vert \leq \vert t_1-t_2\vert$,  $t_1$, $t_2 \geq 0$;%
\label{rho_prop2}
\item $\rho$ is concave so that $\rho(\lambda t_1+(1-\lambda)t_2) 
\geq \lambda \rho(t_1) +(1-\lambda)\rho(t_2)$, $0 \leq \lambda \leq 1$ and 
$t_1$, $t_2 \geq 0$;\label{rho_prop3}
\item $\rho(t)/t$ is a non-increasing function of $t$. \label{rho_prop4}
\end{enumerate}
\end{defn}
We note,  without further comment, that it is sufficient that these 
properties hold for all sufficiently large $t$. 
With this in mind, we may assume that $\rho(0) = 0$. 
While the primary example of a  calibration  function is simply $\rho(t) = t$, 
each of $\rho(t) = t^\alpha$ ($0< \alpha <1$), $\rho(t) = t/\log t$, $\rho(t) = \log t$, 
$\rho(t) =\log\log t$ is a calibration function.

Next we define the lower density associated with a  calibration function $\rho$. 
\begin{defn}The \textsl{lower $\rho$-density\/} of a subset $\A$ of the 
natural numbers $\N$ is defined to be 
\beq\label{rhodensity}
\rhodens(\A) \coloneqq \liminf_{N\to\infty} \frac{\#\{n \in \A \,\colon\,n \leq N\}}{\rho(N)}.
\eeq 
\end{defn}
In the special and most important case when $\rho(t) = t$, 
this is  simply  the \textsl{lower density\/} % and \textsl{upper density\/} 
of $\A$ and abbreviated to 
$\ldens(\A)$. 
The upper $\rho$-density may be defined similarly, though it will not 
be used here. It suffices to note that should the upper and lower $\rho$-densities 
of $\A$ agree then we refer to their common value as the $\rho$-density of $\A$. 
Note also that the lower $\rho$-density of a set $\A$ could well be infinite.

Since an infinite set $\A$ of natural numbers can equivalently be viewed as a 
strictly increasing sequence $(n_k)_{k=1}^\infty$, the density of $\A$ can 
be interpreted in terms of the growth of the sequence. 
Intuitively, $\A$ will have positive lower $\rho$-density if the sequence $(n_k)$ does
not grow too rapidly. 
This is made precise by the following characterisation of 
positive lower $\rho$-density; 
it is a simple extension to general $\rho$ of the well-known 
condition \lq $n_k = \bigO(k)$ as $k \to \infty$\rq\  
that is equivalent to positive lower density,  the case $\rho(t) = t$. 

\begin{lemma}\label{lem:nkdensity}
Let $\A= \{n_k\,\colon\, k \in \N\}$ where $(n_k)_{k=1}^\infty$ 
is a strictly increasing sequence of natural numbers. 
Then $\A$ has positive lower $\rho$-density if and only if 
\beq\label{nkdensity}
\liminf_{k\to\infty} \frac{k}{\rho(n_k)} > 0,
\eeq
equivalently if and only if 
\beq\label{nkdensity2}
\rho(n_k) = \bigO(k) \text{ as } k\to\infty.
\eeq
\end{lemma}

\begin{proof}With $\A = \{n_k\,\colon\, k \in \N\}$, let us suppose that 
\eqref{nkdensity} holds. 
For $n_{k-1} \leq  N < n_k$, we have 
$\#\{n \in \A \,\colon\,n \leq N\} = k-1$. 
Since  $\rho$ is an increasing function, 
\[
\frac{\#\{n \in \A \,\colon\,n \leq N\}}{\rho(N)} = \frac{k-1}{\rho(N)} \geq \frac{k-1}{\rho(n_k)},
\]
for $N$ in this range. It follows that 
\[
\rhodens(\A) 	\geq  \liminf_{k\to\infty} \frac{k-1}{\rho(n_k)} 
			= \liminf_{k\to\infty} \frac{k}{\rho(n_k)} >0,
\]
where the assumption that  $\rho(t) \to \infty$ as $t\to \infty$ was used. 

In the other direction, suppose that $\rhodens(\A) >0$. 
For $N = n_k$, we have 
\[
\frac{\#\{n \in \A \,\colon\,n \leq n_k\}}{\rho(n_k)} = \frac{k}{\rho(n_k)},
\]
from which \eqref{nkdensity} follows. 

It is clear that \eqref{nkdensity} and \eqref{nkdensity2} are equivalent. 
\end{proof}

\subsection{The sequence of natural numbers $\mathbf{(u_k)}$.}
\label{subsection:2.2}
Associated with the  calibration function $\rho$ 
is the sequence of natural numbers $(u_k)_{k=k_0}^\infty$ where 
\beq\label{uk}
u_k = \lfloor \rho^{-1}(k) \rfloor, \quad k\geq k_0, 
\eeq
for a suitably large $k_0$. On occasion, we write $u(k)$ for $u_k$. 
Now, by the ${\rm Lip}(1,1)$ condition on $\rho$, we have that 
$\rho^{-1}(k+1) - \rho^{-1}(k) \geq 1$. 
Since, in general, $\lfloor x \rfloor > x-1$, we deduce that 
\[
u_k \leq \rho^{-1}(k) \leq \rho^{-1}(k+1) -1 
	< \lfloor \rho^{-1}(k+1) \rfloor = u_{k+1} \leq  \rho^{-1}(k+1).
\]
In particular, the sequence $(u_k)$ is increasing. 

In what is essentially a tautology, the sequence $(u_k)$ can be thought of as 
the canonical sequence with positive $\rho$-density. 
To see this, set $\A$ to be the set of natural numbers in the sequence $(u_k)$.
For $u_k \leq N< u_{k+1}$, 
\[
\frac{\#\{n \in \A \,\colon\,n \leq N\}}{\rho(N)} = \frac{k}{\rho(N)} 
	\leq \frac{k}{\rho(u_k)} \leq \frac{k}{\rho(\rho^{-1}(k-1))} = \frac{k}{k-1}
\]
and 
\[
\frac{\#\{n \in \A \,\colon\,n \leq N\}}{\rho(N)} = \frac{k}{\rho(N)} 
	\geq \frac{k}{\rho(u_{k+1})} \geq \frac{k}{\rho(\rho^{-1}(k+1))} = \frac{k}{k+1}.
\]
It follows that the sequence $(u_k)$ has $\rho$-density equal to 1. 

The sequence $(u_k)$ plays a key role in the formulation of a 
$\rho$-Frequent Hypercyclicity Criterion in the next section. 
The following convexity-type property of the sequence $(u_k)$ 
will, in turn, prove to be important in Section~\ref{sec:shifts}.
\begin{lemma}\label{lemma:uk}
The sequence $(u_k)$ is strictly increasing and satisfies
\beq\label{uk_convex}
u_k+ u_j \leq u_{k+j}, \quad k,\, j \geq k_0. 
\eeq
\end{lemma}
\begin{proof}We have just seen that $u_k < u_{k+1}$ for each $k$. It remains
to establish \eqref{uk_convex}. 
In general, $\lfloor x \rfloor + \lfloor y \rfloor \leq \lfloor x+y\rfloor$.
The inequality \eqref{uk_convex} follows if 
\beq\label{rho_property}
\rho^{-1}(k) +\rho^{-1}(j) \leq \rho^{-1}(k+j),
\eeq
since then
\begin{align*}
u_k + u_j & = \lfloor \rho^{-1}(k)\rfloor + \lfloor \rho^{-1}(j)\rfloor\\
& \leq \lfloor \rho^{-1}(k) + \rho^{-1}(j)\rfloor\\
& \leq \lfloor \rho^{-1}(k+j) \rfloor = u_{k+j}.
\end{align*}
The inequality \eqref{rho_property} is a consequence of 
Properties~\eqref{rho_prop3} and \eqref{rho_prop4} in 
Definition~\ref{defn:densityfunction} of a  calibration function. 
For $k\leq j$, and with $\lambda = k/j$, $t_1 = \rho^{-1}(k)$, $t_2 = \rho^{-1}(k+j)$,  
the concavity Property~\eqref{rho_prop3} leads to 
\[
\rho\big[ (k/j)\rho^{-1}(k) + (1-k/j) \rho^{-1}(k+j) \big] \geq  j.
\]
It follows, since $\rho$ is an increasing function, that
\[
(k/j)\rho^{-1}(k) + (1-k/j) \rho^{-1}(k+j) \geq \rho^{-1}(j).
\]
Thus, 
\[
\rho^{-1}(k) + \rho^{-1}(j)  \leq \rho^{-1}(k+j) 
	+ \frac{1}{j}\big[ (k+j)\rho^{-1}(k) - k \rho^{-1}(k+j) \big].
\]
By Property~\eqref{rho_prop4} of $\rho$, and since $\rho$ is increasing, 
the function $\rho^{-1}(s)/s$ is a non-decreasing function of $s$. 
Hence,  $(k+j)\rho^{-1}(k) \leq k \rho^{-1}(k+j)$, and \eqref{rho_property} follows. 
\end{proof}
%
%
% End of Section on $\rho$-density in general 
%
%

\section{$\rho$-frequent hypercyclicity}\label{sec:rhofhc}

We now introduce the main subject of this work which is the concept of 
$\rho$-frequent hypercyclicity.
\begin{defn}\label{defn:rhofhc}
An operator $T$ acting on a separable infinite dimensional Fr\'echet space $X$ is said to be 
\textsl{$\rho$-frequently hypercyclic\/}  if there is an 
element $x$ of $X$ such that, for each nonempty open set $U$ in $X$,
\beq\label{rhofhc}
\rhodens \big( \nTU \big) >0.
\eeq
We call such an $x$ a \textsl{$\rho$-frequently hypercyclic vector\/} for $T$. 
\end{defn}
\noindent The case $\rho(t) = t$ equates to \textsl{frequent hypercyclicity\/} 
as introduced in \cite{BG06}. 
Indeed, $\rho$-frequent hypercyclicity is intended as a spectrum of ever 
stronger notions of hypercyclicity that interpolates between hypercyclicity 
at one end and frequent hypercyclicity at the other.  
In Section~\ref{sec:literature}, we will set out in detail how our notion of 
$\rho$-frequent hypercyclicity overlaps with similar concepts in the 
linear dynamics literature. 

The Hypercyclicity Criterion (cf.\ \cite[Theorem~1.6]{BM09}) is a sufficient condition for an operator to be hypercyclic. 
It has been adapted to cover stronger forms of 
hypercyclicity including frequent hypercyclicity (cf.\ \cite[Theorem~6.18]{BM09}). 
Here we present, in turn, a $\rho$-Frequent Hypercyclicity Criterion. 
We first introduce a more general $\rho$-Frequent Universality Criterion 
for sequences of mappings and deduce the $\rho$-Frequent Hypercyclicity 
Criterion as a special case.

The topological vector space $X$ is said to be 	an \emph{$F$-space} 
if its topology is induced by a complete translation-invariant metric. 
We may assume that this metric is defined by an \emph{$F$-norm}, 
that is a functional $\norm{\, \cdot \,} \colon X \to \R_+$ such that 
for all $x, y \in X$ and scalars $\lambda$,
\begin{enumerate}[label=(\roman*), itemsep=1ex]
	\item $\norm{x +y} \leq \norm{x} + \norm{y}$,
	
	\item $\norm{\lambda x} \leq \norm{x}$, if $\abs{\lambda} \leq 1$,
	
	\item $\displaystyle \lim_{\lambda \to 0} \norm{\lambda x} = 0$,
	
	\item $\norm{x} = 0$ implies that $x = 0$.
\end{enumerate}

We generalise the concept of $\rho$-frequent hypercyclicity to the 
corresponding notion of $\rho$-universality in the following definition 
by considering an arbitrary sequence  of mappings $(T_n)$.

\begin{defn}
Let $X$ and $Y$ be topological spaces and $T_n \colon X \to Y$, 
$n \in \N$, be mappings. 
We say an element $x \in X$ is \emph{$\rho$-frequently universal} 
for the sequence $\T = \left( T_n \right)$ if for every nonempty open 
$U \subset Y$, the return set 
\begin{equation*}
\nT{\T} \coloneqq \{ n \in \N : T_n (x) \in U \}
\end{equation*}
has positive lower $\rho$-density. 
The sequence $(T_n)$ is called \emph{$\rho$-frequently universal} 
if it possesses a $\rho$-frequently universal element.
\end{defn}

Let $X$ be an $F$-space equipped with an $F$-norm $\norm{\, \cdot \,}$. 
A series $\sum_{j=1}^{\infty} x_j$ in $X$ is said to be 
\emph{unconditionally convergent} 
if, for every $\varepsilon >0$, there exists some $N \geq 1$ such that 
for every finite set $F \subset \N$ with 
$F \cap \{1, 2, \dotsc, N\} = \varnothing$ we have
\begin{equation*}
\bigg\Vert \sum_{j \in F} x_j \bigg\Vert < \varepsilon.
\end{equation*}
For an index set $\Lambda$, a collection of series 
$\sum_{j=1}^{\infty} x_{\lambda, j}$, $\lambda \in \Lambda$, in $X$ 
is said to be \emph{unconditionally convergent, uniformly in 
$\lambda \in \Lambda$} if, for every $\varepsilon >0$, 
there exists some $N \geq 1$ such that for every finite set $F \subset \N$ 
with $F \cap \{1, 2, \dotsc, N\} = \varnothing$ and every 
$\lambda \in \Lambda$ we have
\begin{equation*}
\bigg\Vert\sum_{j \in F} x_{\lambda,j} \bigg\Vert< \varepsilon.
\end{equation*}

We will utilise the following Frequent Universality Criterion, 
which is due to Bonilla and Grosse-Erdmann~\cite{BGE07, BGE07_Erratum}.  

\begin{thmA}[Frequent Universality Criterion] \label{thm:FUC}
Let $X$ be an $F$-space, $Y$ a separable $F$-space 
and $T_n \colon X \to Y$, $n \in \N$, continuous linear operators. 
Suppose that there exists a dense subset $Y_0 \subset Y$
and mappings $S_n \colon Y_0 \to X$, $n \in \N$,
such that, for each $y \in Y_0$, the following hold:
\begin{enumerate}[label=\emph{(\Roman*)}, itemsep=1.5ex]
\item $\sum_{j=1}^k T_kS_{k-j}(y)$
	converges unconditionally, uniformly in $k \in \N$;
	\label{item:FUCI}
\item $\sum_{j=1}^\infty T_kS_{k+j}(y)$ converges unconditionally, uniformly in $k \in \N$;
	\label{item:FUCII}
\item $\sum_{j=1}^\infty S_j(y)$ converges unconditionally;
	\label{item:FUCIII}
\item $T_n S_n (y) \to y$ as $n \to \infty$.
	\label{item:FUCIV}
\end{enumerate}
Then the sequence $(T_n)$ is frequently universal.
\end{thmA}
Theorem~\ref{thm:FUC} leads to the following 
$\rho$-Frequent Universality Criterion.

\begin{thm}[$\rho$-Frequent Universality Criterion] \label{thm:rhoFUC}
Let $\rho$ be a calibration function and set, 
as in \eqref{uk},  $u_k \coloneqq \lfloor \rho^{-1}(k) \rfloor$. 
Let $X$ be an $F$-space, $Y$ a separable $F$-space 
and $T_n \colon X \to Y$, $n \in \N$, continuous linear mappings. 
Suppose that there exists a dense subset $Y_0$ of $Y$
and mappings $S_n \colon Y_0 \to X$, $n \in \N$,
such that, for each $y \in Y_0$, the following hold:
\begin{enumerate}[label=\emph{(\roman*)}, itemsep=2ex]
\item $\sum_{j=1}^k T_{u_k} S_{u_{k-j}}(y)$
	converges unconditionally, uniformly in $k\in\N$; \label{item:rhoFUCI}	
\item $\sum_{j=1}^\infty T_{u_k} S_{u_{k+j}}(y)$ converges unconditionally, 
	uniformly in $k \in \N$; \label{item:rhoFUCII}
\item $\sum_{j=1}^\infty S_{u_j} (y)$ converges unconditionally 
	in $X$; \label{item:rhoFUCIII}
\item $T_{u_j} S_{u_j} (y) \to y$ as $j \to \infty$. 
	\label{item:rhoFUCIV}
\end{enumerate}
Then the sequence $(T_n)$ is $\rho$-frequently universal.
\end{thm}

\begin{proof}
Set $\T_n \coloneqq T_{u_n}$ and $\calS_n \coloneqq S_{u_n}$, for $n \in \N$.	
Note that  conditions \emph{\ref{item:rhoFUCI} -- \ref{item:rhoFUCIV}} of the theorem 
imply that  the sequences $(\T_n)$ and $(\calS_n)$ satisfy conditions 
\emph{\ref{item:FUCI} -- \ref{item:FUCIV}}  of Theorem \ref{thm:FUC}, 
and thus $(\T_n)$ is a frequently universal sequence.
Thus there exists $x \in X$ such that, for each nonempty open subset 
$U \subset Y$, there is a strictly increasing sequence $(n_k)$ of positive integers, 
with $n_k = \bigO(k)$, such that $T_{u(n_k)} x \in U$, for $k \in \N$.
	
Next, notice that
\begin{equation*}
\rho \big( u(n_k) \big) 
	= \rho\big( \lfloor \rho^{-1}(n_k) \rfloor \big)
		\leq \rho\big( \rho^{-1}(n_k) \big)
	= n_k = \bigO(k). 
\end{equation*}
Thus the return set $\nT{\T}$ contains the sequence $(u(n_k))_{k\geq 1}$
which, by Lemma~\ref{lem:nkdensity}, has positive lower $\rho$-density.
Hence, $\nT{\T}$ itself has positive lower $\rho$-density and the result follows.  
\end{proof}

A special case of Theorem \ref{thm:rhoFUC} is the following 
$\rho$-Frequent Hypercyclicity Criterion.

\begin{thm}[$\rho$-Frequent Hypercyclicity Criterion] \label{thm:rhoFHCC}
Let $T\colon X \to X$ be a continuous linear operator on a 
separable Fr\'echet space $X$. 
Suppose that there exist a dense subset $X_0 \subset X$ and 
a map $S\colon X_0\to X_0$ such that, for each $x \in X_0$, the following hold:
\begin{enumerate}[label=\emph{($\rho$-\roman*)}, itemsep=1ex]
\item $\sum_{j=1}^k 
	T^{u_k - u_{k-j}} x$ converges unconditionally, uniformly in $k \in \N$;
\item $\sum_{j=1}^\infty 
	S^{u_{j+k} - u_k} (x)$ converges unconditionally, uniformly in $k \in \N$;
\item TS(x) = x.
\end{enumerate}
Then the operator $T$ is $\rho$-frequently hypercyclic. 
\end{thm}
%
%
% End of Section on general $\rho$-frequent hypercyclicity
%
%

\section{$\rho$-frequent hypercyclicity of weighted backward shifts}\label{sec:shifts}

As we will see, the $\rho$-Frequent Hypercyclicity Criterion leads to an explicit 
sufficient condition for a weighted backward shift on $\ell^p(\N)$ to be 
$\rho$-frequently hypercyclic. 
To see this, let $\w = (w_n)_{n=1}^\infty$ be a bounded sequence of positive real numbers and we denote by $\Bw$  the corresponding weighted backward shift defined on $\ell^p(\N)$, 
$1\leq p < \infty$. 
In terms of the canonical basis $(e_n)_{n=0}^\infty$ of $\ell^p(\N)$,
$\Bw$ is defined by $\Bw(e_0) = 0$ and $\Bw(e_n) = w_n e_{n-1}$, $n\geq 1$. 
Thus, 
\[
\Bw(x_0, x_1, x_2, \ldots) = (w_1x_1, w_2x_2, w_3x_3, \ldots).
\]
Applying the $\rho$-Frequent Hypercyclicity Criterion (Theorem~\ref{thm:rhoFHCC}), 
we derive a sufficient condition for $\Bw$ to be $\rho$-frequently hypercyclic. 
For convenience we write 
\[
\alpha_n = \frac{1}{w_1 w_2 \cdots w_n}, \quad n \geq 1.
\]

\begin{thm}\label{thm:sufficient_weighted_shift}
The weighted backward shift $\Bw$ on $\ell^p(\N)$, $1\leq p < \infty$, 
is $\rho$-frequently hypercyclic if, for each $i \in \N$, the series 
\beq\label{Bw_criterion}
\sum_{j=1}^\infty \alpha_{i + u_{j+k} - u_k}^{\,p}
\mbox{ 
is convergent, uniformly in $k \in\N$. }
\eeq
\end{thm}

\begin{proof}
As is standard, we take $X_0$ to be the set $D$ of finitely supported sequences,
which is a dense subset of  $\ell^p(\N)$. 
Define $S \colon D\to D$ by 
\[
Se_n = \frac{1}{w_{n+1}}\, e_{n+1}, \quad n \geq 0, 
\]
and we extend to $D$ by linearity. It is immediate that $\Bw Se_n =e_n$ for $ n \geq 0$, thus
showing that ($\rho$-iii) of Theorem~\ref{thm:rhoFHCC} is satisfied. 

To see that ($\rho$-i) is satisfied, note that for each sequence $x$ in $D$ 
there corresponds a positive integer $N = N(x)$ such that $\Bw^n x = 0$ for $n > N(x)$.
Thus, 
\[
\sum_{j=1}^k  \Bw^{u_k - u_{k-j}} x
\]
is, for each $k$, a sum of terms selected from 
$\Bw x$, $\Bw^2 x$, $\ldots$ $\Bw^{N(x)}x$, 
of which there are only finitely many such sums possible. 
Unconditional convergence, uniformly in $k$, is therefore automatic. 

In short, ($\rho$-i) and ($\rho$-iii) hold automatically irrespective 
of the sequence of weights $\w = (w_n)_{n=1}^\infty$. 

Next we interpret ($\rho$-ii) in the context of a weighted backward shift. 
We need, for each $x$ in $D$, that 
\[
\sum_{j=1}^\infty S^{u_{j+k} - u_k} x 
\]
converges unconditionally in $\ell^p(\N)$ and uniformly in $k \in \N$.
Since each $x$ in $D$ is a finite linear combination of basis elements,
it is sufficient that this  holds for each basis element $e_i$ individually. 
With $x = e_i$, $i\geq 0$, 
\[
S^j e_i  = \frac{1}{w_{i+1}w_{i+2} \cdots w_{i+j}}\, e_{i+j}\\
 = \left(  \prod_{r=1}^i w_r \right) \alpha_{i+j} e_{i+j} = c_i \alpha_{i+j} e_{i+j}
\]
where $c_i = \prod_{r=1}^i w_r$ depends only on $i$. 
This leads to 
\[
\Big\Vert \sum_{j=1}^\infty S^{u_{j+k} - u_k} e_i\Big\Vert_{p} = 
	c_i  \Big\Vert \sum_{j=1}^\infty \alpha_{i +u_{j+k} - u_k} e_{i +u_{j+k} - u_k}  
		\Big\Vert_{p}.
\]
Since, with $k$ (and $i$) fixed, $i +u_{j+k} - u_k$ is a strictly increasing 
sequence in $j$ [as $(u_j)_{j=1}^\infty$ is a strictly increasing sequence], 
each term in this last sum corresponds to a different basis element. 
Hence, 
\[
\Big\Vert \sum_{j=1}^\infty S^{u_{j+k} - u_k} e_i\Big\Vert_{p}^p = 
	c_i^{\, p} \sum_{j=1}^\infty \alpha_{i +u_{j+k} - u_k}^{\,p}.
\]
In short, Condition~($\rho$-ii) of Theorem~\ref{thm:rhoFHCC} is equivalent, 
in the case of the weighted backward shift $\Bw$ on $\ell^p(\N)$, to \eqref{Bw_criterion}.
\end{proof}
Let us be explicit regarding  what the condition \eqref{Bw_criterion}  
requires:  for each $i$ and to each $\epsilon$ positive there corresponds 
a natural number $N$ such that, for every $k$, 
\beq\label{Bw_criterion2}
\sum_{j =N}^\infty  \alpha_{i+ u_{j+k} - u_k}^{\, p} \leq \epsilon.
\eeq
The sufficient condition \eqref{Bw_criterion2} for $\Bw$ to be 
$\rho$-frequently hypercyclic certainly implies that 
$\inf_{n} \alpha_n = 0$, which is the characterisation, due to 
Salas \cite{Salas}, for $\Bw$ to be hypercyclic. 

\smallskip With $\rho(t) = t$, so that $u_k = k$ for each $k$, the condition 
\eqref{Bw_criterion} reduces to  
\beq\label{Bw_fhc_criterion}
\sum_{j=1}^\infty \alpha_j^{\,p} < \infty.
\eeq
In their seminal paper, Bayart and Grivaux \cite[Example~2.7]{BG06} showed that  
\eqref{Bw_fhc_criterion} is sufficient for $\Bw$ to be frequently hypercyclic. 
It was not until  2015 that the necessity was established by Bayart and Ruzsa 
\cite{BR15} (who also established the corresponding necessary and sufficient metric condition
for frequent hypercyclicity of weighted bilateral shifts acting on the space $\ell^p(\Z)$). 
Their proof rests, in part, on an extension of a famous result of Erd\H{o}s and S\'ark\"ozy
to the effect that if $\A$ is a subset of $\N$ with positive upper density then the set $\A - \A$ 
is syndetic, in that it has bounded gaps. Equivalently, the
set of $k$ for which $B_k = \A \cap (\A-k)$ is nonempty is syndetic. 
In \cite{BR15} they extend this result by proving that the set of $k$ 
for which $B_k$ has a uniformly positive upper density is also syndetic.  

It is possible to provide a much simpler argument that establishes the necessity of 
the condition \eqref{Bw_criterion} for $\Bw$  to be 
$\rho$-frequently hypercyclic on $\ell^p(\N)$ in the case of  a restricted class of 
weighted backward shifts.  
While this suffices to construct the examples in the following section, the question 
remains  whether the condition \eqref{Bw_criterion}  is necessary 
for $\rho$-frequent hypercyclicity of a general weighted backward shift
on $\ell^p(\N)$. 
In terms of the sequence $(u_k)_{k=1}^\infty$ from Subsection~\ref{subsection:2.2}, 
which we view as the canonical sequence with positive lower $\rho$-density, 
we have the following result. 

\begin{thm}\label{thm:Bw_criterion}
Let  $\Bw$  be a weighted backward shift on $\ell^p(\N)$, $1\leq p <\infty$, 
for which the weights satisfy $w_n \geq1$ for each $n$. 
Then, $\Bw$ is $\rho$-frequently hypercyclic 
if and only if 
\beq\label{Bw_criterion_simpler}
\sum_{j=1}^\infty \alpha_{u_j}^{\,p} < \infty.
\eeq
\end{thm}

We first need a standard result that corresponds, both in its statement 
and its proof, to \cite[Proposition~9.17]{LinearChaos2011}
(cf.\ also \cite[Section~3]{BR15}).

\begin{prop}\label{prop:4.1}
Let  $\Bw$  be a weighted backward shift on $\ell^p(\N)$, $1\leq p <\infty$. 
If $\Bw$ is $\rho$-frequently hypercyclic then there is a sequence $ (n_j)_{j=1}^\infty$
of positive lower $\rho$-density for which 
\beq\label{nk_series}
\sum_{j=1}^\infty \alpha_{n_j}^{\,p} < \infty.
\eeq
\end{prop}
\begin{proof}
Let $x$ be a $\rho$-frequently hypercyclic vector for $\Bw$. Taking $U$ to be the ball 
in $\ell^p(\N)$ with centre $2e_0$ and radius $1$, 
there is a sequence  $ (n_j)_{j=1}^\infty$ of positive lower $\rho$-density such that 
\[
\Bw^{n_j}x = (w_1w_2 \cdots w_{n_j}x_{n_j}, \ldots ) \in U, \mbox{ for each } j. 
\]
But then $\vert x_{n_j} / \alpha_{n_j} - 2\vert < 1$, which leads to 
$\vert x_{n_j} \vert > \alpha_{n_j}$ for each $j$. 
Since $x$ is in $\ell^p(\N)$, we have 
\[
\sum_{j=1}^\infty \alpha_{n_j}^{\,p} \leq \sum_{j=1}^\infty \vert x_{n_j}\vert^{\,p} 
	\leq \Vert x \Vert_p^{\,p} < \infty.
\qedhere
\]
\end{proof}

The final ingredient in the proof of Theorem~\ref{thm:Bw_criterion} is the following 
result on series that uses  Cauchy's Condensation Test. 
In the special case $\rho(t) = t$, we see that if $(a_n)_n$ is a positive decreasing sequence 
with $\sum_j a_{n_j}$ convergent  where $n_j = O(j)$, then  the full series $\sum_j a_j$ is 
also convergent. 

\begin{prop}\label{prop:4.2}
Suppose that $(a_n)_{n=1}^\infty$ is a non-increasing sequence of positive 
real numbers. 
If  $\A = (n_j)_{j=1}^{\infty}$ has positive lower $\rho$-density 
and  $\sum_{j=1}^\infty a_{n_j}< \infty$, then $\sum_{j=1}^\infty a_{u_j} < \infty$.
\end{prop}

\begin{proof}Since $\A$ has positive lower $\rho$-density, it follows from Lemma~\ref{lem:nkdensity} that 
$\rho(n_j) = O(j)$. 
Hence, we may take it that $\rho(n_j) \leq 2^r j$ for some positive integer $r$, this for each $j$. 
Equivalently, since $n_j$ is an integer, $n_j \leq \lfloor \rho^{-1}(2^r j) \rfloor$. 
In particular, 
\[
n(2^j) \leq  \lfloor \rho^{-1}(2^{r+j})\rfloor= u(2^{r+j}).
\]
By the assumed convergence of the series $\sum_{j=1}^\infty a_{n_j}$ and by Cauchy's Condensation Test, 
\[
\sum_{j=1}^\infty 2^j a_{n(2^j)} < \infty.
\]
Since the sequence $(a_n)$ is non-increasing, $a_{n(2^j)} \geq a_{u(2^{r+j})}$. 
It follows that  
\[
\sum_{j=r+1}^\infty 2^j a_{u(2^j)}   = \sum_{j=1}^\infty 2^{r+j} a_{u(2^{r+j})} 
	\leq  2^{r} \sum_{j=1}^\infty 2^j a_{n(2^j)} < \infty.
\]
Hence, again by Cauchy's Condensation Test, $\sum_{j=1}^\infty a_{u_j} < \infty$.  
\end{proof}

\begin{proof}[Proof of Theorem~\ref{thm:Bw_criterion}]
Since the weights satisfy $w_n\geq1$ for each $n$, the sequence $(\alpha_n)_{n=1}^\infty$ is 
non-increasing. In \eqref{Bw_criterion}, therefore, 
$\alpha_{i + u_{j+k} - u_k} \leq \alpha_{u_{j+k} - u_k}$. 
Moreover, by  \eqref{uk_convex}, 
$u_{j+k}-u_k \geq u_j$ and so, for each $i$ and $k$,
\[
\sum_{j=1}^\infty \alpha_{i + u_{j+k} - u_k}^{\,p} 
	\leq \sum_{j=1}^\infty \alpha_{u_{j+k} - u_k}^{\,p}
	\leq \sum_{j=1}^\infty \alpha_{u_j}^{\,p}.
\]
That is, \eqref{Bw_criterion_simpler} implies \eqref{Bw_criterion}. 
Since \eqref{Bw_criterion} is a general sufficient condition for the
$\rho$-frequent hypercyclicity of $\Bw$, 
 condition \eqref{Bw_criterion_simpler} is a sufficient condition for 
$\rho$-frequent hypercyclicity of $\Bw$ if all the weights are greater than or equal to 1. 

To establish the necessity of the condition \eqref{Bw_criterion_simpler} for this 
restricted family of weighted backward shifts, Proposition~\ref{prop:4.1} 
provides the existence of a sequence  $ (n_j)_{j=1}^\infty$
of positive lower $\rho$-density for which 
$\sum_{j=1}^\infty \alpha_{n_j}^{\,p}$ is convergent.
We then deduce from Proposition~\ref{prop:4.2}, 
with $a_n = \alpha_n^p$, that 
$\sum_{j=1}^\infty \alpha_{u_j}^{\,p}$ is convergent, which is \eqref{Bw_criterion_simpler}.
\end{proof}
%
%
% End of Section on $\rho$-frequent hypercyclicity 
%	of weighted backward shifts
%
%

\section{$\rho$-frequent hypercyclicity as $\rho$ varies}
\label{sec:varying_rho}

It is natural to ask if the notion of $\rho$-frequent hypercyclicity depends, 
as it should do, on the choice of calibration function $\rho$. 
Theorem~~\ref{thm:Bw_criterion} is used in the  construction of 
examples of weighted backward shifts that demonstrate 
this dependence. 
If $\rho_1(t) \gg \rho_2(t)$ as $t \to \infty$, we would like an example of 
an operator that is  $\rho_2$-frequently hypercyclic but not 
$\rho_1$-frequently hypercyclic. 
For simplicity, the weighted shifts constructed live on $\ell^1(\N)$ rather than a general 
$\ell^{\,p}(\N)$. 
% Theorem
\begin{thm}\label{thm:rho1_rho2}
Suppose that both $\rho_1$ and $\rho_2$ are calibration functions and 
that $\rho_1(t) = \psi(\rho_2(t))$, where $\psi(t) / t$  is an increasing function of $t$ that is unbounded as $t \to \infty$. 
There then exists a bounded sequence of weights $\w =(w_n)_{n=1}^\infty$ such that:
\begin{enumerate}[label=(\roman*), itemsep=1ex]
\item $\Bw$ is  $\rho_2$-frequently hypercyclic on $\ell^1(\N)$;
\item $\Bw$ is not $\rho_1$-frequently hypercyclic on $\ell^1(\N)$.
\end{enumerate}
\end{thm}
\noindent Thus, for example, there is a bounded weighted backward shift on $\ell^1(\N)$ that is 
$(t/\log t)$-frequently hypercyclic but not frequently hypercyclic. 

\smallskip\noindent The proof of Theorem~\ref{thm:rho1_rho2} requires the following  fact regarding series.
\begin{lemma}\label{lem:series}
Let $(a_n)_{n=1}^\infty$ be an increasing, unbounded sequence of 
positive real numbers. 
Then, there exists a sequence $(b_n)_{n=1}^\infty$ of positive real numbers 
such that 
\begin{align}
\frac{1}{2} b_{n-1} \leq b_n \leq b_{n-1},
	\label{series1}\\
\sum_{n=1}^\infty a_n\,b_n =  \infty,
	\label{series2}
	\shortintertext{and}
\sum_{n=1}^\infty b_n< \infty.
	\label{series3}
\end{align}
\end{lemma}
\begin{proof}
We first produce a decreasing sequence $(\tilde b_n)_{n=1}^\infty$ 
that satisfies \eqref{series2} and \eqref{series3} and then 
adjust it so that the modified sequence decreases at most geometrically. 
We choose an increasing sequence of positive integers $(n_k)_{k=0}^\infty$
and set, for $k \geq 1$, 
\[
A_k = \big\{ n \in\N \colon n_{k-1} \leq a_n < n_k\big\}.
\]
The sequence $(n_k)$ is constructed inductively. We set $n_0 = a_1$
and $n_1 = a_i$ where $i$ is the first time that $a_i > a_1$.
Once $n_k$ is chosen, we choose $n_{k+1}$ so that $A_{k+1}$ is nonempty, 
(i) $n_{k+1} \geq 2^{k+1}$ and 
%(ii) $n_k (\#A_{k+1}) \geq n_{k-1} (\# A_k)$.
(ii) $\#A_{k+1} \geq \# A_k$.
This is possible since $(a_n)_{n=1}^\infty$ is an increasing 
and unbounded sequence.  

For each $k$, and for $n$ in $A_k$, we set 
\[
\tilde b_n = \frac{1}{n_{k-1} (\#A_k)}. 
\]
The condition (ii) above, together with $n_k \geq n_{k-1}$, 
ensures that the sequence $(\tilde b_n)_{n=1}^\infty$ is non-increasing. 
Since $n_k \geq 2^k$, 
\beq\label{btilde1}
\sum_{n=1}^\infty \tilde b_n = \sum_{k=1}^\infty \left(\, \sum_{n \in A_k} \tilde b_n\right)
	= \sum_{k=1}^\infty \frac{1}{n_{k-1}} \leq \sum_{k=1}^\infty 2^{-k+1} < \infty.
\eeq
On the other hand, 
\begin{align}
\sum_{n=1}^\infty a_n \tilde b_n & = \sum_{k=1}^\infty \left(\, \sum_{n \in A_k} a_n \tilde b_n\right)
		= \sum_{k=1}^\infty \left( \frac{1}{n_{k-1} (\#A_k)} \sum_{n \in A_k} a_n\right)
			\nonumber\\
	& \geq \sum_{k=1}^\infty \left( \frac{1}{n_{k-1} (\#A_k)} \sum_{n \in A_k}n_{k-1}\right)
		= \sum_{k=1}^\infty 1 =  \infty. \label{btilde2}
\end{align}
To guarantee \eqref{series1}, we replace the sequence $(\tilde b_n)_{n=1}^\infty$ by a 
sequence $(b_n)_{n=1}^\infty$ of possibly larger terms 
while retaining properties \eqref{series2} and \eqref{series3}. 
We set $b_1 = \tilde b_1$ and construct $b_n$ inductively by
\[
b_n = \max \big\{ \tfrac{1}{2} b_{n-1}, \tilde b_n \big\}, \quad n \geq 2.
\]
Since $b_n \geq \tilde b_n$, \eqref{series2} follows directly from \eqref{btilde2}.
For \eqref{series1}, we have $b_n \geq \tfrac{1}{2}b_{n-1}$ by construction. 
To see that the sequence $(b_n)_{n=1}^\infty$ is non-increasing, 
there are two cases: if $b_n = \tfrac{1}{2}b_{n-1}$ then certainly $b_n \leq b_{n-1}$, 
while if $b_n = \tilde b_n$ then $b_n = \tilde b_n \leq \tilde b_{n-1} \leq b_{n-1}$
using, here, that the sequence $(\tilde b_n)_{n=1}^\infty$ is non-increasing. 
This verifies \eqref{series1}.

Finally, we verify \eqref{series3}, the convergence of the series $\sum_n b_n$.
Let  $(m_k)_{k=1}^\infty$ be the strictly increasing sequence of those positive integers
for which $b_{m_k} = \tilde b_{m_k}$. Thus, $b_n = \tfrac{1}{2} b_{n-1}$ if 
$n \not= m_k$ for some $k$. 
In words, the sequence $(b_n)$ decays exactly geometrically as $n$ runs 
between successive integers $m_k$. 
For each $k$, therefore, 
\[
\sum_{n = m_k}^{m_{k+1}-1} b_n = \tilde b_{m_k} + \tfrac{1}{2} \tilde b_{m_k} 
	+ \big(\tfrac{1}{2})^2 \tilde b_{m_k} + \cdots 
		+ \big(\tfrac{1}{2})^{m_{k+1}-m_k} \tilde b_{m_k}
\leq 2\, \tilde b_{m_k}.
\]
Hence, 
\[
\sum_{n=1}^\infty b_n \leq \sum_{k=1}^\infty \left( \sum_{n = m_k}^{m_{k+1}-1} b_n \right)
	\leq 2\sum_{k=1}^\infty \tilde b_{m_k} \leq 2 \sum_{n=1}^\infty \tilde b_n,
\]
and this last sum is finite by \eqref{btilde1}. That is, \eqref{series3} holds for 
the adjusted sequence $(b_n)_n$. 
\end{proof}
\begin{proof}[Proof of Theorem~\ref{thm:rho1_rho2}]
Set  $u_{1,k} = \lfloor \rho_1^{-1}(k) \rfloor$ and
$u_{2,k} = \lfloor \rho_2^{-1}(k) \rfloor$ for each positive integer $k$. 
We construct a bounded sequence of weights $\w = (w_n)_{n=1}^\infty$, 
with $w_n \geq 1$ for each $n$, such that 
\beq\label{rho1rho2condition}
\sum_{j=1}^\infty \alpha_{u_{1,j}} =  \infty \mbox{ and }
	\sum_{k=1}^\infty \alpha_{u_{2,k}} < \infty.
\eeq
By Theorem~\ref{thm:Bw_criterion}, the corresponding weighted shift
$\Bw$ is $\rho_2$-frequently hypercyclic on $\ell^1(\N)$ 
but not $\rho_1$-frequently hypercyclic on $\ell^1(\N)$.

Most weights $w_n$ will equal 1, the only exception being when $n = u_{2,k}$ 
for some $k$ in which case the weights have  yet to be chosen. 
A priori, these weights must be at least 1.

Since $\rho_1$ grows much faster than $\rho_2$, the sequence 
$\{u_{2,k}\}_{k=1}^\infty$ will be much thinner than the sequence $\{u_{1,j}\}_{j=1}^\infty$.
For fixed (and large) $k$, we need a lower bound on the number of integers $u_{1,j}$ 
that lie between $u_{2,k}$ and $u_{2,k+1}$. If $j$ is an integer with 
\beq\label{star1}
\rho_2^{-1}(k) \leq \rho_1^{-1}(j) \leq \rho_2^{-1}(k+1).
\eeq
then $u_{2,k} \leq u_{1,j} \leq u_{2,k+1}$. Thus, it is sufficient to estimate  
how many integers $j$ satisfy \eqref{star1}. 
Since $\rho_1^{-1} = \rho_2^{-1}\circ \psi^{-1}$, and 
since  both $\rho_2$ and then $\psi$ are increasing, 
the inequalities \eqref{star1} are equivalent to 
$\psi(k) \leq j \leq \psi(k+1)$. 

Now, 
\begin{align*}
\psi(k+1) - \psi(k) 
	& = \frac{\psi(k)}{k}\bigg[(k+1) \left(\left.\frac{\psi(k+1)}{k+1}\right/ \frac{\psi(k)}{k}\right)  - k \bigg]\\
	& \geq \frac{\psi(k)}{k}  \big[ (k+1) - k \big]=\frac{\psi(k)}{k},
\end{align*}
where the assumption that $\psi(t)/t$ is increasing was used. 
In general, there are at least $\lfloor d \rfloor$ integers in an interval $[x,y]$ of length $d$. 
We therefore conclude that, for large $k$, 
\[
\#\big\{ j \colon u_{2,k} \leq u_{1,j} < u_{2,k+1} \big\} 
	\geq \Big\lfloor \frac{\psi(k)}{k} \Big\rfloor - 1\eqqcolon a_k. 
\]
This is illustrated in Figure \ref{fig1}.

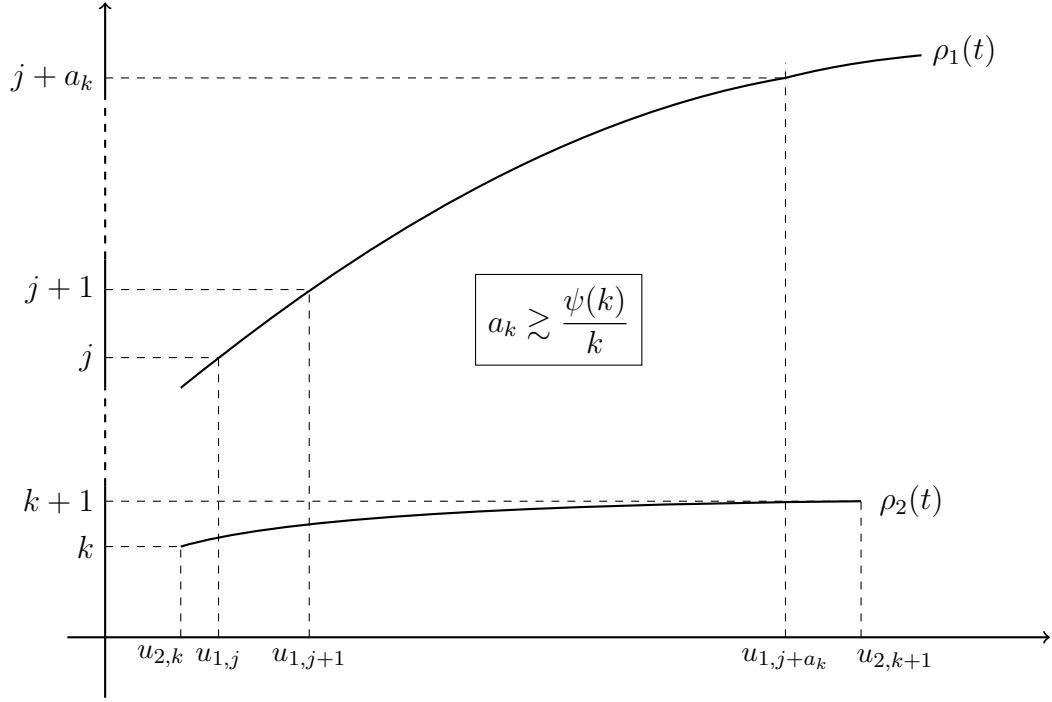
\begin{figure}
	\begin{tikzpicture}[scale=1.0]
		
		% Axes
		\draw[->,thick] (-0.5,-0.2) -- (12.5,-0.2);
		%\draw[->,thick] (0,-1.0) -- (0,8.2);
		\draw[thick] (0,-1.0) -- (0,1.8);
		\draw[dashed, thick] (0,1.8) -- (0,3.1);
		\draw[thick] (0,3.1) -- (0,4.8);
		\draw[dashed,thick] (0,4.8) -- (0,7);
		\draw[->,thick] (0,7) -- (0,8.2);
		
		% Key x-coordinates
		\coordinate (uik)      at (1.1,-0.2);
		\coordinate (uij)      at (1.9,-0.2);
		\coordinate (uijp1)    at (2.4,-0.2);
		\coordinate (uijp2)    at (2.7,-0.2);
		\coordinate (u_end1)   at (9.0,-0.2);
		\coordinate (u_end2)   at (10.0,-0.2);
		
		% Horizontal reference levels
		\coordinate (k)        at (0,1.0);
		\coordinate (kp1)      at (0,1.6);
		\coordinate (j)        at (0,3.5);
		%\coordinate (jp1)      at (0,3.9);
		\coordinate (jp2)      at (0,4.4);
		\coordinate (top)      at (0,7.2);
		
		% Labels on y-axis
		\node[left] at (k)   {$k$};
		\node[left] at (kp1) {$k+1$};
		\node[left] at (j)   {$j$};
		%\node[left] at (jp1) {$j+1$};
		\node[left] at (jp2) {$j+1$};
		\node[left] at (top) {$j+a_k$};
		
		% Dashed horizontal guides
		\draw[dashed] (0,1.0) -- (1,1.0);
		\draw[dashed] (0,1.6) -- (10.0,1.6);
		\draw[dashed] (0,3.5) -- (1.5,3.5);
		%\draw[dashed] (0,3.9) -- (2.4,3.9);
		\draw[dashed] (0,4.4) -- (2.7,4.4);
		\draw[dashed] (0,7.2) -- (9.0,7.2);
		
		% rho_2
		\draw[thick]
		(1,1.0)
		.. controls (2.2,1.35) and (5.0,1.55) ..
		(10.0,1.6);
		
		\node[right] at (10.1,1.6) {$\rho_2(t)$};
		
		% rho_1
		\draw[thick]
		(1,3.1)
		.. controls (2.5,4.3) and (5.5,6.6) ..
		(9.0,7.2)
		.. controls (9.7,7.38) and (10.3,7.45) ..
		(10.8,7.5);
		
		\node[right] at (10.8,7.55) {$\rho_1(t)$};
		
		% Vertical dashed guides
		\draw[dashed] (1,-0.2) -- (1,1.0);
		\draw[dashed] (1.5,-0.2) -- (1.5,3.5);
		%\draw[dashed] (2.4,-0.2) -- (2.4,3.9);
		\draw[dashed] (2.7,-0.2) -- (2.7,4.4);
		
		\draw[dashed] (9.0,-0.2) -- (9.0,7.4);
		\draw[dashed] (10.0,-0.2) -- (10.0,1.6);
		
		% x-axis labels
		\node[below, left] at (1.2,-0.45) {\small$u_{2,k}$};
		\node[below] at (1.5,-0.2) {\small$u_{1,j}$};
		%\node[below] at (2.4,-0.2) {\small$u_{1,j+1}$};
		\node[below] at (2.7,-0.2) {\small$u_{1,j+1}$};
		
		\node[below] at (9.0,-0.2) {\small$u_{1,j+a_k}$};
		\node[below, right] at (9.8,-0.5) {\small$u_{2,k+1}$};
		
		% Annotation box
		\node[
		draw,
		minimum width=2.2cm,
		minimum height=1.2cm
		] at (6.0,4.0)
		{$a_k \gtrsim \dfrac{\psi(k)}{k}$};
	\end{tikzpicture}
	\caption{On the interval $[u_{2,k}, u_{2,k+1}]$  of the $x$-axis, on which 
		$\rho_2$ increases by 1 from $k$ to $k+1$, the faster growing function $\rho_1$ passes 
		through each of the integers from $j$ to $j+a_k$ for which the estimate 
		$a_k \gtrsim \psi(k)/k$ holds. Here $\psi=\rho_1\circ\rho_2^{-1}$.}
				\label{fig1}
\end{figure}

By assumption, the sequence $(a_k)_k$ is increasing and  unbounded,
so that Lemma~\ref{lem:series} applies. 
Thus, there is a sequence $(b_n)_{n=1}^\infty$ of positive real numbers 
that satisfies \eqref{series1}, \eqref{series2} and \eqref{series3} 
with respect to the sequence $(a_k)_k$; we take $b_0 = 1$ for convenience. 
We are now in a position to write down the weights $(w_n)_{n=1}^\infty$. 
We set 
\beq\label{weights}
w_n = \begin{cases}
b_{n-1}/b_n, & n =  u_{2,k} \mbox{ for some }k;\\
1, &\mbox{ otherwise}.
\end{cases}
\eeq 
By \eqref{series1}, we have $1 \leq w_n \leq 2$ for each $n$. 
In particular, $\Bw$ is bounded on $\ell^1(\N)$.
We need to verify \eqref{rho1rho2condition} with this choice of weights. 
For $n=u_{2,k}$ for some $k$, we have 
\[
\alpha_{u_{2,k}} = 1/\prod_{n=1}^{u_{2,k}} w_n 
	= 1/\prod_{i=1}^k w_{u_{2,i}}
	= 1/\prod_{i=1}^k \big( b_{i-1}/b_i \big) = b_k.
\]
As a consequence, $\sum_{k=1}^\infty \alpha_{u_{2,k}} < \infty$. 

Finally, we verify that $\sum_{j=1}^\infty \alpha_{u_{1,j}}$ is divergent. 
For fixed $k$ and for  $u_{2,k} \leq u_{1,j}  < u_{2,k+1}$,
\[
\alpha_{u_{1,j}} = 1/\prod_{n=1}^{u_{1,j}} w_n  = 1/\prod_{i=1}^k w_{u_{2,i}} = b_k.
\]
Hence, 
\[
\sum_{j: u_{2,k} \leq u_{1,j}  < u_{2,k+1}} \alpha_{u_{1,j}}  
	= b_k \times \big(\#\big\{ j \colon u_{2,k} \leq u_{1,j} < u_{2,k+1} \big\}\big) \geq b_ka_k,
\]
and so $\sum_{j=1}^\infty \alpha_{u_{1,j}} \geq \sum_{k=1}^\infty b_k a_k = \infty$.
\end{proof}

\section{Related formulations of the concept of  frequent hypercyclicity} 
\label{sec:literature}

 It is incumbent on us to clarify the relationship between the notion of 
$\rho$-frequent hypercyclicity as set out herein and similar, indeed overlapping, 
notions that are already to be found in the literature. 
Bonilla and Grosse-Erdmann~\cite{BGE07} observed that frequent hypercyclicity 
can be re-formulated as follows: 
$x$ is a frequently hypercyclic vector for $T$ if and only if 
the return set $\nTU$ of each nonempty open set $U$ in $X$ can be enumerated 
as an infinite strictly increasing sequence $(n_k)$ for which 
$n_k = \bigO(k)$ as $k\to\infty$: this also follows from \eqref{nkdensity2} with $\rho(t) = t$. 
Bayart and Matheron~\cite{BM09weak} adopt this viewpoint when they propose 
a generalisation of frequent hypercyclicity in the course of their study of how 
\lq frequently hypercyclic\rq\  an operator can be without being weakly mixing.

\begin{defn}
Let $(m_k)_{k\geq 1}$ be a strictly increasing sequence of natural numbers and let 
 $(n_k)_{k=1}^\infty$ be a strictly increasing sequence that is
\emph{$(m_k)$-bounded\/}, that is  $n_k = \bigO(m_k)$ as $k\to\infty$.
An operator $T$ is said to be \textsl{$(m_k)$-frequently hypercyclic\/}%
\footnote{Bayart and Matheron's own terminology is simply \lq $(m_k)$-hypercyclic\rq, 
but we prefer to add the adjective \lq frequently\rq.}  
%($(m_k)$-fhc)
if and only if there is an element $x$ in the space for which the return set $\nTU$ 
of each nonempty open set $U$ in $X$ can be enumerated as an infinite strictly increasing sequence $(n_k)$ that is $(m_k)$-bounded. 
\end{defn}

The above  definition is described in \cite{BM09weak} as 
\lq the most obvious generalisation of frequent hypercyclicity\rq, 
nonetheless $\rho$-frequent hypercyclicity could  also be considered a reasonable generalisation
of this concept. It is, perhaps, a matter of preference. 
If $(m_k)$ is any increasing sequence of natural numbers such that 
$\lim_{k\to \infty} m_k/k = \infty$, in \cite[Theorem~1.2]{BM09weak} they
construct an operator on $\ell^1(\N)$ that is $(m_k)$-frequently hypercyclic, 
but not weakly mixing.

Intuitively, in $\rho$-frequent hypercyclicity we count how often (at least about $\rho(N)$ times)
the orbit of $x$ up to time $N$ lies in a given open set $U$, 
whereas in $(m_k)$-frequent hypercyclicity
we count how long we need to wait (at most about $m_k$ iterates) until the orbit of 
$x$ has landed in the open set $k$ times. These are therefore dual notions. 
However, they are not equivalent - the enemy here is if either $\rho$ 
grows slowly or $m_k$ grows quickly. 
This distinction was also pointed out by Kosti\'c \cite{Kostic2019} 
(see Remark~1 and the example preceding it) who also gave a positive 
result similar to Proposition~\ref{prop:6.1}.  
The distinction is between the conditions $\rho(n_k) = O(k)$ and 
$n_k = O(m_k)$ with $m_k = \lfloor \rho^{-1}(k) \rfloor$. 

\begin{example}
Set $\rho(t) = \log_2 t$. 
We claim that there is no strictly increasing sequence of positive integers 
$(m_k)$ with the property that a strictly increasing sequence of positive integers $(n_k)$ 
has positive lower $\log_2$-density if and only if it is $(m_k)$-bounded. 

Suppose that $(m_k)$ was such a sequence. Let $j$ be an arbitrary fixed positive integer. 
The sequence $(n_k) = (2^{(j+1)k})$ has, by Lemma~\ref{lem:nkdensity},
positive lower $\log_2$-density since, in this case, $\rho(n_k) = (j+1)k =O(k)$. 
The sequence should therefore be $(m_k)$-bounded; there should be a finite 
constant $C_j$ such that $2^{(j+1)k} \leq C_j m_k$, each $k$. Then,
\[
m_k \geq 2^{jk} \mbox{ for all sufficiently large } k.
\]
Therefore, for each $j$, 
\[
\liminf_{k\to\infty} \frac{k}{\log_2 m_k} \leq \liminf_{k\to\infty} \frac{k}{\log_2(2^{jk})}
= 1/j.
\]
As a consequence of Lemma~\ref{lem:nkdensity}, $(m_k)$ has zero lower $\log_2$-density. 
But $(m_k)$ itself is  $(m_k)$-bounded and so it should have 
positive lower $\log_2$-density.   \hfill $\square$
\end{example}

\begin{example}As a dual example, set $(m_k) = (2^k)$. 
We claim that there is no calibration function $\rho$ for which positive 
lower $\rho$-density is equivalent to $(2^k)$-boundedness for every 
strictly increasing sequence $(n_k)$.

Suppose that $\rho$ was such a calibration function.
First choosing $(n_k) = (m_k) = (2^k)$, we would have that $(2^k)$ has positive lower 
$\rho$-density so that, by Lemma~\ref{lem:nkdensity}, 
$\rho(2^k) \leq Ck$ for some finite $C$ and for all large $k$. 
Now  take $(n_k) = (2^{2k})$. This sequence 
also has positive lower $\rho$-density by Lemma~\ref{lem:nkdensity}, since
\[
\liminf_{k\to\infty} \frac{k}{\rho(n_k)} = \liminf_{k\to\infty} \frac{k}{\rho(2^{2k})} 
\geq \liminf_{k\to\infty} \frac{k}{2Ck} = \frac{1}{2C}.
\]
However, the sequence $(2^{2k})$ is not $(2^k)$-bounded. \hfill $\square$  
\end{example}

In many cases, however, the properties of positive lower $\rho$-density and 
of being $(m_k)$-bounded are equivalent, that is they are 
satisfied by the same sequences $(n_k)$. 

\begin{prop}\label{prop:6.1}%
\rm{\textbf{(a)}} Suppose that $\rho$ is a calibration function that satisfies, 
for some finite $C_1$ and all sufficiently large $t$, the additional assumption
\beq
\rho^{-1}(2t) \leq C_1\, \rho^{-1}(t)\label{eq:rhodoubling1}
\eeq
Set $m_k \coloneqq \lfloor \rho^{-1}(k) \rfloor$. 
Then a strictly increasing sequence $(n_k)$ has positive 
lower $\rho$-density if and only if it is $(m_k)$-bounded. 
In particular, an operator $T$ is $\rho$-frequently hypercyclic if and only 
if it is $(m_k)$-frequently hypercyclic. 

\medskip

\noindent\rm{\textbf{(b)}} Suppose that $(m_k)$ is a strictly increasing sequence of natural numbers
that satisfies,  for some finite $C_2$ and all sufficiently large $k$,  
the additional conditions
\begin{align}
m_k & \leq \tfrac{1}{2}\big(m_{k-1} +m_{k+1}\big), \label{eq:mkdoubling1}\\
m_{2k} &\leq C_2\,m_k. \label{eq:mkdoubling2}
\end{align}
Define $\rho$ by $\rho(m_k) = k$ for each $k$, and $\rho$ 
piecewise linear between $m_k$ and $m_{k+1}$ for each $k$. 
Then $\rho$ is a calibration function and, moreover, 
a strictly increasing sequence $(n_k)$ is $(m_k)$-bounded 
if and only if it has positive lower $\rho$-density. 
In particular, an operator $T$ is $(m_k)$-frequently hypercyclic if and only 
if it is $\rho$-frequently hypercyclic. 
\end{prop}

\begin{proof}(a) Suppose first that $\rho$ is a calibration function that has 
property \eqref{eq:rhodoubling1} in addition to those in 
Definition~\ref{defn:densityfunction}.
Set $m_k = \lfloor \rho^{-1}(k) \rfloor$, and suppose that $(n_k)$  is a 
strictly increasing sequence with positive lower $\rho$-density. 
By Lemma~\ref{lem:nkdensity}, the latter implies that there is a constant $C$ 
such that $\rho(n_k) \leq Ck$ for each $k$. 
We may take $C$ to be $2^j$ for some natural number $j$.
Then, for each $k$, 
\[
n_k \leq \rho^{-1}(2^j k) \leq  C_1^j\rho^{-1}(k) \leq C_1^j  (m_k +1) \leq 2 C_1^j\,m_k,
\]
which implies  that the sequence $(n_k)$ is $(m_k)$-bounded. 
The second inequality comes from \eqref{eq:rhodoubling1}.

Conversely, suppose that $(n_k)$ is a strictly increasing $(m_k)$-bounded sequence. 
Then, there is a constant $C$ such that $n_k \leq Cm_k \leq C \rho^{-1}(k)$ for each 
$k$. Take $C = 2^j$ for some natural number $j$ and note that, 
by Property (iv) of Definition~\ref{defn:densityfunction}, $\rho(2t) \leq 2\rho(t)$. 
Then, 
\[
\rho(n_k) \leq \rho\big( 2^j \rho^{-1}(k) \big) \leq 2^j \rho\big( \rho^{-1}(k) \big) = 2^jk,
\]
and so $(n_k)$ has positive lower $\rho$-density, again by Lemma~\ref{lem:nkdensity}.

\medskip

(b) Suppose that $(m_k)$ is 
a strictly increasing sequence of natural numbers that satisfies 
\eqref{eq:mkdoubling1} and \eqref{eq:mkdoubling2}. 
With $\rho$ defined by $\rho(m_k) = k$ for each $k$, and $\rho$ 
piecewise linear between $m_k$ and $m_{k+1}$, we find that   
$\rho$ satisfies the requirements of Definition~\ref{defn:densityfunction}, 
and so is a calibration function. 
Certainly, $\rho$ is continuous, strictly increasing and unbounded since 
the sequence $(m_k)$ is increasing. 
Set $\Delta m_k = m_k - m_{k-1}$. 
The slope of the line segment on the graph of $\rho$ between 
$(m_{k-1},k-1)$ and $(m_k, k)$ is $1/\Delta m_k$. 
Since $\Delta m_k \geq 1$ for each $k$, this slope is at most 1 and  
so $\rho$ is Lip$(1,1)$. 
Since, by \eqref{eq:mkdoubling1}, $\Delta m_k$ is non-decreasing in $k$, 
the concavity of $\rho$ follows. 

To complete the verification that $\rho$ has the properties of a calibration
function listed in Definition~\ref{defn:densityfunction}, 
we show that $\rho(t)/t$ is a non-increasing function of $t$. 
Note that $\rho(m_k)/m_k = k/m_k$ and that, with $m_0=0$, 
\[
m_k = \sum_{j=1}^k \Delta m_j \leq k\, \Delta m_k \leq k\,\Delta m_{k+1} 
	= k ( m_{k+1} - m_k ).
\]
Here, we again used the assumption \eqref{eq:mkdoubling1} 
that $\Delta m_k$ is non-decreasing in $k$.
It follows that $k/m_k$ is non-increasing in $k$. For $m_k \leq t \leq m_{k+1}$, 
the expression $\rho(t)/t$ decreases from $k/m_k$ to $(k+1)/m_{k+1}$. 
Hence, $\rho(t)/t$ is non-increasing in $t$ and so $\rho$ is a bona fide calibration
function. 

The calibration function $\rho$ also satisfies \eqref{eq:rhodoubling1}. 
In fact, suppose that $t = k+x$ with $k \in \N$ and $0 \leq x\leq 1$. 
Since $\rho^{-1}$ is linear between $k$ and $k+1$, 
\[
\rho^{-1}(t) = \rho^{-1}(k) + x\big[ \rho^{-1}(k+1) - \rho^{-1}(k) \big]
	= (1-x)\,m_k + x\,m_{k+1}.
\]
Since $\Delta m_{2k+1} \leq \Delta m_{2k+2}$, the function $\rho^{-1}$ is convex on $[2k,2k+2]$, 
and so 
\begin{align*}
\rho^{-1}(2t) & = \rho^{-1}(2k+2x)\\
&\leq \rho^{-1}(2k) + 2x \bigg( \frac{\rho^{-1}(2k+2)-\rho^{-1}(2k)}{2} \bigg)\\
& = (1-x)m_{2k} + x \, m_{2k+2}\\
& \leq C_2 \big[ (1-x) m_k + x \,m_{k+1} \big] = C_2\, \rho^{-1}(t). 
\end{align*}
The assumption \eqref{eq:mkdoubling2} was used at the last step. 

At this point, we are in the setting of Part (a) of this proposition and 
may conclude from Part (a) that an increasing of natural numbers $(n_k)$ 
is $(m_k)$-bounded if and only if it has positive lower $\rho$-density. 
\end{proof}

The study of  $(m_k)$-frequent hypercyclicity is progressed further in the work of 
Kosti\'c \cite{Kostic2019} and of Heo, Kim and Kim \cite{HeoKimKim2017}. 
In particular, both papers consider the notion of \lq$q$-frequent hypercyclicity\rq\ 
first introduced by Gupta and Mundayadan \cite{GupMun1} (see also \cite{GupMun2}).
This notion equates to $(m_k)$-frequent hypercyclicity for the sequence 
$m_k = k^q$ where $q$ is a positive integer 
(or $q$ is positive and real as in \cite{Kostic2019}). 
Equivalently, an operator $T$ on a Fr\'echet space $X$ is said to be 
\textsl{$q$-frequently hypercyclic\/} ($q\geq 1$) if there is an 
element $x$ in $X$ such that, for each nonempty open set $U$ in $X$,
\beq\label{q-fhc}
\liminf_{N\to\infty} \frac{\#\{n \in \nTU \,\colon\,n \leq N^q\}}{N} >0.
\eeq
Such an $x$ is called a \textsl{$q$-frequently hypercyclic vector\/}  for $T$. 
Proposition~\ref{prop:6.1} applies and shows that the notion is equivalent to 
$\rho$-frequent hypercyclicity with $\rho(t) = t^{1/q}$.  
Gupta and Mundayadan prove a $q$-frequent hypercyclicity criterion 
\cite[Theorem~3.7]{GupMun1} and show, by an argument specific
to this case,  that if $p> q$, $p$, $q$ in $\N$, 
then there is a weighted backward shift $\Bw$ on $\ell^1(\N)$ that is 
$p$-frequently hypercyclic but not $q$-frequently hypercyclic. 
This is the special case of our Theorem~\ref{thm:rho1_rho2} with 
$\rho_1(t) = t^{1/q}$, $\rho_2(t) = t^{1/p}$ and $\psi(t) = t^{p/q}$. 
They also prove that, as in the case $q = 1$ of frequent hypercyclicity, 
rotations and powers of $q$-frequently hypercyclic operators remain 
$q$-frequently hypercyclic. 

There are a number of criteria for  $(m_k)$-frequent hypercyclicity 
in the literature, albeit expressed in slightly different forms. 
The first such criterion, in the case of $q$-frequent 
hypercyclicity, is due to Gupta and Mundayadan 
\cite{GupMun1, GupMun2}, see also Heo et al.\ \cite{HeoKimKim2017}.
Kim  \cite{Kim} provides an $(m_k)$-frequent hypercyclicity criterion. 
In the setting of various more general densities and more general
sequences $(m_k)$, Kosti\'c 
derives an $(m_k)$-frequent hypercyclicity criterion -- see 
Theorem~3.1 in \cite{Kostic2019} and the preceding discussion.

\begin{example} The following is an example of 
a weighted backward shift $\Bw$ on $\ell^1(\N)$ 
that is $\log_2$-frequently hypercyclic  but not $q$-frequently hypercyclic
for any $q\geq 1$. A straightforward modification will produce 
an example on $\ell^p(\N)$, $1 \leq p < \infty$.

\smallskip\noindent Set $w_0 = w_1 = w_2 = 1$ and 
\[
w_n = \left( \frac{\log_2 n}{\log_2 (n-1)} \right)^2, \quad n \geq 3. 
\]
All weights satisfy $w_n \geq1$ and so Theorem~\ref{thm:Bw_criterion} applies.  
Here, 
\[
\alpha_n = \frac{1}{\big( \log_2 n \big)^2}, \quad n \geq 3.
\]
Then, $\sum_{n=3}^\infty \alpha_{2^n} = \sum_{n=3}^\infty 1/n^2$ is convergent 
and so $\Bw$ is $\log_2$-frequently hypercyclic. 
On the other hand, for each $q\geq 1$, 
\[
\sum_{n=3}^\infty \alpha_{n^q}
	= \sum_{n=3}^\infty \frac{1}{ \big( \log_2 n^q \big)^2}
	= \sum_{n=3}^\infty \frac{1}{q^2 \big( \log_2 n \big)^2},
\]
which is a divergent series. 
% (since, for instance, we have $\log_2n \leq \sqrt{n}$ for $n$ large).
That is, $\sum_{n=3}^\infty \alpha_{n^q}$ is divergent and so, 
by Theorem~\ref{thm:Bw_criterion}, $\Bw$ is not $t^{1/q}$-frequently hypercyclic. 
\qed
\end{example}

\smallskip

We also mention here  another useful perspective for studying the `frequency' of frequently hypercyclic operators, which was introduced by Ernst and Mouze~\cite{EM19, EM21}. Their approach is to replace the  lower density in \eqref{density} with weighted (lower) densities.
To this end, they consider a  sequence $(a_k)_{k \geq 1}$ of nonnegative numbers, with $\sum_{k=1}^n a_k \to \infty$ as $n\to \infty$. This determines the associated \emph{admissible matrix} $A = (a_{n,k})$, which is given by the coefficients 
\begin{align*}
	a_{n,k} = 
	\begin{cases}
		\displaystyle 
		\frac{a_k}{\sum_{j=1}^{n} a_j}, & 1 \leq j \leq n, \\
		0, & k >n.
	\end{cases}
\end{align*}
The \emph{lower $A$-density} of the subset $E \subset \N$, with respect to the admissible matrix $A = (a_{n,k})$, is then defined as
\begin{equation*}
	\ldens_A(E) = \liminf_{n\to\infty} \sum_{k=1}^{\infty} a_{n,k} \1_E(k).
\end{equation*}

So for an admissible matrix $A = (a_{n,k})$ and a separable infinite dimensional Fréchet space $X$, the operator $T \colon X \to X$ is said to be $A$\emph{-frequently hypercyclic} if there exists $x \in X$ such that for any nonempty open set $U \subset X$, the return set $\nTU$ has positive lower $A$-density.

The original definition of frequent hypercyclicity can be recovered by taking $a_k = 1$ for each $k \in \N$ in the sequence $(a_k)$. This gives the well-known Ces\`{a}ro matrix that corresponds to the `natural' lower density in \eqref{density}.
Among the other densities considered in \cite{EM19, EM21} is the density determined by the sequence $(a_k)_k = (1/k)_{k \in \N}$.  This gives the so-called \emph{lower logarithmic density} of $E \subset \N$, which is defined as
\begin{equation} \label{logFHC}
	\ldens_A(E) \coloneqq \liminf_{n\to\infty} \frac{1}{\log n} \sum_{\substack{k \in E, \\ k \leq n}} \frac{1}{k}.
\end{equation}

Using this framework, the authors in \cite{EM19, EM21} utilise various scales of lower $A$-densities to perform a quantitative study of the Frequent Hypercyclicity Criterion. They also identify interesting examples. These including operators that are frequently hypercyclic but not $A$-frequently hypercyclic on a particular scale, and an operator that is \emph{logarithmically frequently hypercyclic}, as per \eqref{logFHC}, but not frequently hypercyclic.
Theorem \ref{thm:rho1_rho2} can be viewed an analogue of these results.

\smallskip 

Finally, we mention that the collection of subsets of the natural numbers $\N$ 
that have positive lower $\rho$-density forms a proper 
Furstenberg family. That $T$ is $\rho$-frequently hypercyclic 
is then a special case of $\mathcal{A}$-hypercyclicity as defined by 
B\`es et al.\ \cite{BesMenetPerisPuig2016} or of 
$\F$-hypercyclicity as defined by Kosti\'c \cite{Kostic2019}. 

%
%% amsplain
%\bibliographystyle{abbrv}
%
%\bibliography{fhcbib}

\end{document}